\documentclass[12pt]{amsart}
\usepackage{amssymb,amsmath,amsfonts}
\usepackage{enumerate}
\usepackage{graphicx, color}
\usepackage[letterpaper, margin=1.2in]{geometry}

\allowdisplaybreaks

\newcommand{\N}{\mathbb N}
\newcommand{\R}{\mathbb R}
\newcommand{\Z}{\mathbb Z}
\newcommand{\Q}{\mathbb Q}
\newcommand{\C}{\mathbb C}
\newcommand{\rar}{\rightarrow}
\newcommand{\mc}{\mathcal}

\newtheorem*{theorem*}{Theorem}
\newtheorem{theorem}{Theorem}
\newtheorem{lemma}[theorem]{Lemma}

\begin{document}

\title[Higher moments of distances between consecutive Ford spheres]{Higher moments of distances between\\consecutive Ford spheres}
\author{Alan Haynes, Kayleigh Measures}


\begin{abstract}
In previous work the second author derived an asymptotic formula for the sum of the distances between centers of consecutive Ford spheres. In this paper we extend these results by proving asymptotic formulas for higher moments of the distances. Our proofs rely on lattice point counting estimates with error terms coming from the Gauss circle problem.
\end{abstract}

\maketitle

\section{Introduction}
Over the last few decades there has been a surge of interest in attempting to understand finer properties of the Farey sequence and its generalizations. Much of this work is motivated by well known connections between Farey fractions and the Generalized Riemann Hypothesis \cite{Fran1924,Huxl1971,Land1924}, although a good amount of research has also arisen in other contexts, for example in the study of billiards \cite{BocaZaha2007} and dynamics in spaces of Euclidean lattices \cite{Mark2010}.

In 1938 L.~R.~Ford published the simple but elegant paper \cite{Ford1938}, in which he explained a geometric construction, using Farey fractions, of what are now referred to as Ford circles. In today's terminology, Ford's construction gives an example of an integral Apollonian circle packing between two lines. Ford circles appear naturally in the study of continued fractions, and they have found several interesting applications, for example to the construction of the contour of integration in the proof of Rademacher's formula for the partition function \cite[Chapter 5]{Apos1990}.

Recent work by Chaubey, Malik, and Zaharescu \cite{ChauMaliZaha2015} demonstrated how tools from analytic number theory can be used to give accurate asymptotic estimates for integral moments of distances between consecutive Ford circles whose centers lie above fixed horizontal lines. In \cite{Meas2018} the second author began a generalization of these results to a complex analogue of Ford circles, Ford spheres, which were also introduced by Ford in his 1938 paper. The main result of \cite{Meas2018} was an asymptotic formula for the sum of the distances between consecutive Ford spheres with centers lying above fixed horizontal planes. In this paper we generalize these results by proving asymptotic formulas for the higher integral moments.

Let $I_2$ denote the unit square in the upper right quadrant of the complex plane and for $S\in\N$ define
\begin{equation*}
\mathcal{G}_S=\left\{\frac{r}{s}\in I_2 ~:~ r,s\in\Z[i],~ (r,s)=1,~ |s|\le S  \right\}.
\end{equation*}
The Ford sphere corresponding to a point $r/s\in\mathcal{G}_S$ is a sphere of radius $1/2|s|^2$ in $\C\times\R^+$, tangent to the boundary of this region, with center at $(r/s,1/2|s|^2).$ As shown in \cite[Section 8]{Ford1938}, the Ford spheres corresponding to any two points $r/s$ and $r'/s'$ in $\mathcal{G}_S$ are either tangent, or they do not intersect. If they are tangent then, following Ford's convention, we say that the fractions $r/s$ and $r'/s'$ are adjacent. Furthermore, we say that two fractions $r/s, r'/s'\in\mathcal{G}_S$ are consecutive in $\mathcal{G}_S$ if they are adjacent and if there is at least one other fraction in $\Q[i]$, with corresponding Ford sphere of radius less than $1/2S^2$, which is adjacent to both $r/s$ and $r'/s'$. This is a natural geometric generalization of the notion of consecutive Farey fractions of a given order (see the discussion surrounding \cite[Definition 2.6]{Meas2018}).

Now for $k,S\in\N$ let
\begin{equation*}
  \mc{M}_k(S)=\sum_{\substack{\frac{r}{s},\frac{r'}{s'}\in\mc{G}_S\\\text{consecutive}}}\left(\frac{1}{2|s|^2}+\frac{1}{2|s'|^2}\right)^k,
\end{equation*}
where the sum is over all unordered pairs of fractions $r/s$ and $r'/s'$ which are consecutive in $\mathcal{G}_S$. These quantities, the $k$th moments of distances between centers of spheres corresponding to consecutive fractions in $\mathcal{G}_S$, are the precise analogues for Ford spheres of the moments for Ford circles considered in \cite{ChauMaliZaha2015}. It was proven in \cite{Meas2018} that, for any $\epsilon >0$,
\begin{equation*}
\mc{M}_1(S)=\frac{\pi}{4}\zeta_i^{-1}(2)(2C-1)S^2+O_\epsilon(S^{1+\epsilon}),
\end{equation*}
where $\zeta_i$ is the Dedekind zeta function for $\Q(i)$ and
\begin{equation*}
C=-\int_0^{1/\sqrt{2}}\ln \left(\sqrt{2}u\right)\left(1-u^2\right)^{1/2}~\mathrm{d}u.
\end{equation*}
The method of proof in \cite{Meas2018} does not immediately generalize to $k>1$, because the error terms in the corresponding calculations become of the same order of magnitude as what would otherwise be considered the main terms. However, by more refined lattice point counting techniques we are now able to prove the following result.
\begin{theorem}\label{thm.kthMom}
  For each integer $k\ge 2$, there exists a constant $\xi_k>0$ with the property that, for any $\epsilon>0$ and for any $S\in\N$,
  \begin{equation*}
    \mc{M}_k(S)=\xi_kS+O_{\epsilon}(S^{2\kappa+\epsilon}),
  \end{equation*}
  with $\kappa\le 131/416$.
\end{theorem}
Throughout this paper we will use the quantity $\kappa$ to denote the smallest exponent which can be used in the error term of the generalized Gauss circle problem, as presented in \cite{Huxl2003}. Note that in some references to this problem the quantity we are calling $\kappa$ is replaced by $2\kappa$. The best known upper bound for $\kappa$ was obtained by Huxley \cite{Huxl2003}, who showed that it is no larger than $131/416$. We discuss some details of this problem below.

\section{Preliminary results}

\subsection{Notation}
First we introduce our notation, which is for the most part the same as that used in \cite{Meas2018}. We will denote the set of Gaussian integers $q$ for which $\mathrm{Re}(q)>0$ and $\mathrm{Im}(q)\ge 0$ by $\Z[i]^+$. For any $q\in\Z[i]$, we use the convention that $\sum_{d|q}$ denotes a sum over $d\in\Z[i]^+$ which divide $q$.

Any $q\in\Z[i]$ can be written uniquely in the form
\begin{equation*}
q=u\cdot p_1^{\alpha_1}p_2^{\alpha_2}\cdots p_k^{\alpha_k},
\end{equation*}
where $u\in\{\pm 1,\pm i\}$ and the $p_i$ are prime elements in $\Z[i]$, with $p_i\not=p_j$ for $i\not=j,$ $\mathrm{Re}(p_i)>0,$ and $\mathrm{Im}(p_i)\ge 0$. Using this representation, we define $\mu_i:\Z[i]\rar\Z$, the M\"{o}bius function for the Gaussian integers, by
\begin{equation*}
\mu_i(q)=\begin{cases}
1 &\text{if}~q=u,\\
(-1)^k &\text{if}~\alpha_1=\cdots=\alpha_k=1,\\
0 &\text{otherwise}.
\end{cases}
\end{equation*}
We also define $\phi_i:\Z[i]\rar\N$, the Euler-phi function for the Gaussian integers, by
\begin{equation*}
\phi_i(q)=\left|\left(\Z[i]/q\Z[i]\right)^*\right|.
\end{equation*}
As shown in \cite{Meas2018}, these functions satisfy the expected elementary number theoretic identities, namely that for any $q\in\Z[i]$,
\begin{equation*}
\sum_{d|q}\phi_i(d)=|q|^2,\qquad\sum_{d|q}\mu_i(d)=\begin{cases}
1 &\text{if}~q~\text{is a unit},\\
0 &\text{otherwise},
\end{cases}
\end{equation*}
and
\begin{equation*}
\frac{\phi_i(q)}{|q|^2}=\sum_{d|q}\frac{\mu_i(d)}{|d|^2}.
\end{equation*}

The sum of squares function $r_2:\N\rar\Z$ is defined by
\begin{equation*}
r_2(n)=\#\{(a,b)\in\Z^2~:~a^2+b^2=n\}.
\end{equation*}
As observed by Gauss, the summatory function of $r_2(n)$ can be estimated trivially by the asymptotic formula
\begin{equation*}
\sum_{n=1}^Nr_2(n)=\pi N+O(N^{1/2}).
\end{equation*}
The Gauss circle problem is the problem of obtaining best possible upper bounds for the quantities
\begin{equation*}
\left|\sum_{n=1}^Nr_2(n)-\pi N\right|,
\end{equation*}
as $N\rar\infty$. One of the first major non-trivial results in this direction was published in 1906 by Sierpinski \cite{Sier1906} (see also \cite{Schi1972}), who showed that
\begin{equation*}
\sum_{n=1}^Nr_2(n)=\pi N+O(N^{1/3}).
\end{equation*}
Subsequently many authors have made substantial contributions to this problem (see the introduction to \cite{Huxl2003} for a more extensive list), and the currently best known upper bound, due to Huxley \cite{Huxl2003}, is given by
\begin{equation*}
\sum_{n=1}^Nr_2(n)=\pi N+O\left(N^{131/416}(\log N)^{18627/8320}\right).
\end{equation*}
Huxley's proof applies equally well to regions which are homothetic dilations and translations of convex regions with sufficiently smooth boundaries. As mentioned in the introduction, throughout this paper we will use the constant $\kappa$ to denote the smallest number with the property that the error terms in such lattice point estimates are $\ll N^\kappa$. It is known, by work attributed in the literature to Hardy and also to Landau, that $\kappa>1/4$, and it follows from the above mentioned result of Huxley that $\kappa$ is no larger than $131/416$.

The Dedekind zeta function for the number field $\Q(i)$ is defined by the Dirichlet series
\begin{equation*}
\zeta_i(s)=\sum_{q\in\Z[i]^+}\frac{1}{|q|^{2s}}.
\end{equation*}
Using the main term in the asymptotic formula for the summatory function of $r_2(n)$, it is not difficult to show that this series converges absolutely and uniformly in compact subsets of the region $\mathrm{Re}(s)>1$. It also follows from the corresponding Euler product representation that, for $\mathrm{Re}(s)>1$,
\begin{equation*}
\zeta_i^{-1}(s)=\sum_{q\in\Z[i]^+}\frac{\mu_i(q)}{|q|^{2s}}.
\end{equation*}

\subsection{Elementary lemmas} Here we gather together a few lemmas which will be useful in the proof of our main result.
\begin{lemma}\label{lem.PowSum2}
  For $R\ge 1$ we have that
  \begin{equation*}
    \sum_{\substack{s\in\Z[i]\\0<|s|\le R}}\frac{1}{|s|^2}=2\pi\log R +O(1).
  \end{equation*}
\end{lemma}
\begin{proof}
  Using partial summation, we have that
  \begin{align*}
    \sum_{0<|s|\le R}\frac{1}{|s|^2}&=\sum_{\ell\le R^2}\frac{r_2(\ell)}{\ell}\\
    &=\sum_{\ell\le R^2}\frac{1}{\ell(\ell+1)}\sum_{j=1}^\ell r_2(j)+\frac{1}{R^2+1}\sum_{j\le R^2}r_2(j)\\
    &=\sum_{\ell\le R^2}\frac{\pi \ell+O(\ell^{\kappa})}{\ell(\ell+1)}+\frac{\pi R^2+O(R^{2\kappa})}{R^2+1}\\
    &=\pi\sum_{\ell\le R^2}\frac{1}{\ell+1}+O(1)\\
    &=2\pi\log R +O(1).
  \end{align*}
\end{proof}

\begin{lemma}\label{lem.PowSumTail1}\cite[Lemma 3.6]{Meas2018} For $k>1$ and $R\ge 1$ we have that
\begin{equation*}
  \sum_{|s|\ge R}\frac{1}{|s|^{2k}}\ll_k\frac{1}{R^{2(k-1)}}.
\end{equation*}
\end{lemma}

\begin{lemma}\label{lem.SumPhii} \cite[Lemma 3.7]{Meas2018} For $R\ge 1$ we have
\begin{equation*}
  \sum\limits_{\substack{s\in \Z[i]^+ \\ |s| \leq R}} \phi_i(s) = z_1R^4 + O\left(R^{2+2\kappa}\right),
\end{equation*}
with $z_1=\frac{\pi}{8} \zeta_i^{-1}(2).$
\end{lemma}

\begin{lemma}\label{lem.PhiPowSum2} There is a constant $z_2>0$ such that, for $R\ge 1$,
\begin{equation*}
  \sum_{\substack{s\in\Z[i]^+\\|s|\le R}}\frac{\phi_i(s)}{|s|^4}=z_1\log R+\frac{z_1+z_2}{4}+O(R^{2\kappa-2}),
\end{equation*}
where $z_1$ is as above and
\begin{equation*}
z_2=\int_1^\infty\left(\sum\limits_{\substack{s\in \Z[i]^+ \\ |s| \leq t^{1/4}}} \phi_i(s)-z_1t\right)t^{-2}~\mathrm{d}t.
\end{equation*}
\end{lemma}
\begin{proof}
  The proof of this lemma is included in the proof of \cite[Proposition 5.2]{Meas2018}.
\end{proof}

\begin{lemma}\label{lem.PhiPowSum3} For each $k> 2$ there is a constant $z_k>0$ such that, for $R\ge 1$,
\begin{equation*}
  \sum_{\substack{s\in\Z[i]^+\\|s|\le R}}\frac{\phi_i(s)}{|s|^{2k}}=z_k+O_k\left(\frac{1}{R^{2(k-2)}}\right).
\end{equation*}
\end{lemma}
\begin{proof}
  Since $\phi_i(s)\le |s|^2$, it follows from Lemma \ref{lem.PowSumTail1} that
  \begin{equation*}
    z_k=\sum_{|s|\in\Z[i]^+}\frac{\phi_i(s)}{|s|^{2k}}
  \end{equation*}
  is positive and finite. By the same result, we also have that
  \begin{align*}
    \sum_{\substack{s\in\Z[i]^+ \\ |s|\le R}}\frac{\phi_i(s)}{|s|^{2k}}&=z_k-\sum_{\substack{s\in\Z[i]^+\\|s|> R}}\frac{\phi_i(s)}{|s|^{2k}}\\
    &=z_k+O_k\left(\frac{1}{R^{2(k-2)}}\right).
  \end{align*}
\end{proof}

\begin{lemma}[\textbf{Abel's summation formula}]\label{lem.AbelSumForm} Suppose we have functions $a:\N\rar\R$ and $f:\R\rar\R$, and that $f'(x)$ exists and is continuous. Let $A(x)=\sum_{n\le x}a(n)$. Then
\begin{equation*}
\sum_{n\le x} a(n)f(n)=A(x)f(x)-\int_1^xA(t)f'(t)~\mathrm{d}t.
\end{equation*}
\end{lemma}

\subsection{Geometric criterion for consecutivity}\label{subsec.Consec} In this subsection we recall a geometric criterion, introduced in \cite[Section 4]{Meas2018}, which is useful for evaluating summations over pairs of Gaussian integers $s$ and $s'$ which occur as denominators of consecutive fractions in $\mathcal{G}_S$. First of all, it is proved in \cite[Lemma 4.3]{Meas2018} that two Gaussian integers $s$ and $s'$ will appear as consecutive denominators in $\mathcal{G}_S$ if and only if all of the following conditions are satisfied:
\begin{itemize}
	\item[(i)] $|s|,|s'|\le S$,
	\item[(ii)] $(s,s')=1$, and
	\item[(iii)] $|s'+us|>S$ for some $u\in\{\pm 1, \pm i\}$.
\end{itemize}
Furthermore, if all of these conditions are satisfied, then there are exactly four distinct pairs $r,r'\in\Z[i]$ for which $r/s$ and $r'/s'$ are consecutive in $\mathcal{G}_S$.

For $S\ge 1$ and $s\in\mathcal{G}_S$, let $\Omega_s$ denote the set of points in the complex plane which lie inside the circle of radius $S$ centered at the origin, and outside at least one of the circles of radius $S$ centered at the points $\pm s$ and $\pm is$. This region is illustrated in Figure 1. The definition of the region depends on $S$, but for aesthetic purposes we suppress this dependence in our notation.

\begin{figure}[h]\label{fig.Omega_s}
	\centering
	\def\svgwidth{0.8\columnwidth}
	\includegraphics{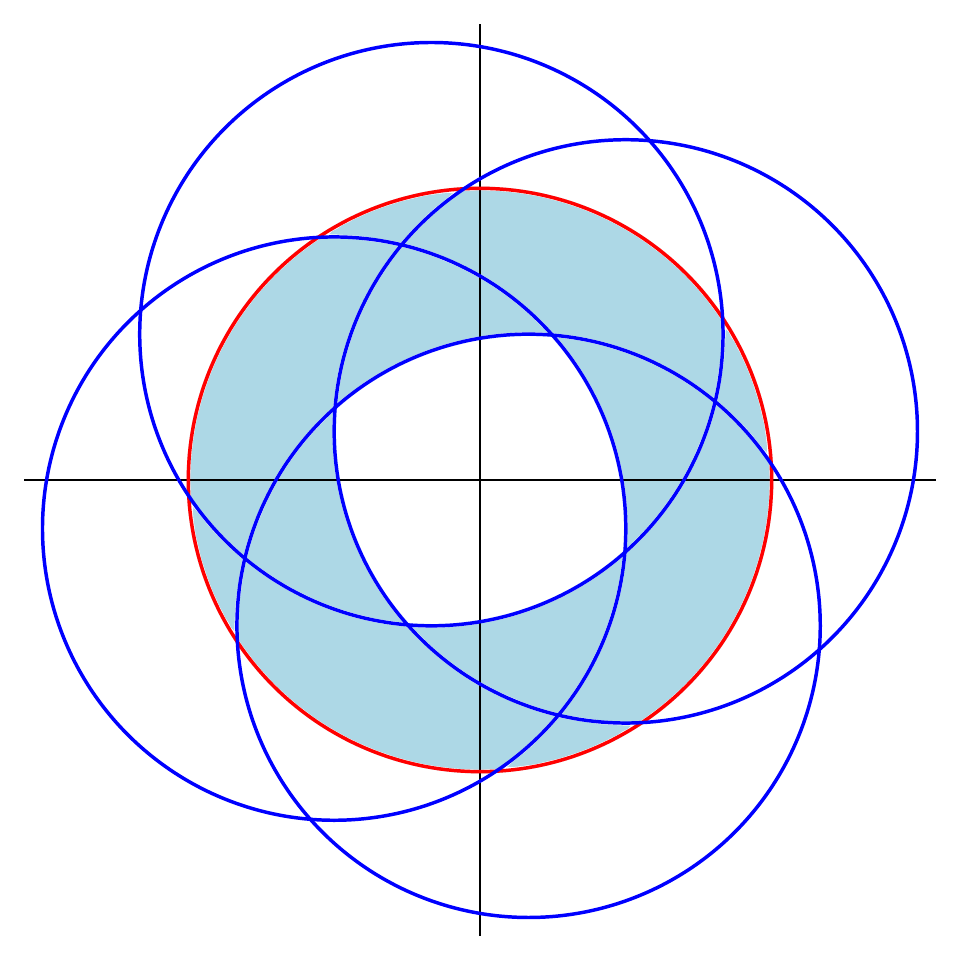}
	\caption{The $5$ circles each have radius $S$, and they are centered at the points $0, \pm s$ and $\pm is.$ The shaded region is $\Omega_s$}
\end{figure}

The characterization of consecutive denominators given in the first paragraph of this subsection shows that there is a natural bijective correspondence between the collection of points $s'\in\Omega_s\cap\Z[i]$ which satisfy $(s,s')=1,$ and the collection of fractions which are consecutive in $\mathcal{G}_S$ to a fraction with denominator $s$.

In the calculations which follow we will need a formula for the area of $\Omega_s$ which captures the way in which it depends on $|s|$ and $S$. For fixed $S$ we define a function $I_1:[0,\sqrt{2}S]\rar\R$ by
\begin{equation*}
I_1(t)=8S^2\int_0^{\sin^{-1}\left(\frac{t}{\sqrt{2}S}\right)}\cos^2u~\mathrm{d}u.
\end{equation*}
Then, as proved in \cite[Proposition 5.1]{Meas2018}, the area of $\Omega_s$ is given by the formula
\begin{equation}\label{eqn.OmegaArea}
|\Omega_s|=I_1\left(|s|\right)-2|s|^2.
\end{equation}
This formula will be useful to us in the next sections.

\section{Proof of Theorem \ref{thm.kthMom}, $k=2$ case}
First we deal with the $k=2$ case of Theorem \ref{thm.kthMom}. The general case, which will be explained in the next section, is similar, but there are a few differences in the calculations. We begin by expanding the square in the summand defining $\mathcal{M}_2(S)$ to write it as
\begin{align}
  \mc{M}_2(S)=\sum_{\frac{r}{s}\in\mc{G}_S}\frac{\#\{r'/s'\in \mc{G}_S ~\text{consecutive to}~ r/s\}}{4|s|^4}+\sum_{\substack{\frac{r}{s},\frac{r'}{s'}\in\mc{G}_S\\
  \text{consecutive}}}\frac{1}{2|s|^2|s'|^2}=\Sigma_1+\Sigma_2.\label{eqn.2ndMomAsym1}
\end{align}
It is important to keep in mind that the sum defining $\mathcal{M}_2(S)$ is over unordered consecutive pairs $r/s$ and $r'/s'$ in $\mathcal{G}_S$, which is why there is $4$ in the denominator of $\Sigma_1$ and not a $2$. 

The sum $\Sigma_2$ will end up being asymptotically smaller than our main error term, so we estimate it first. From the characterization of consecutivity provided in Section \ref{subsec.Consec} we know that if a pair $s,s'$ occurs as a pair of consecutive denominators in $\mathcal{G}_S$ then there are four possible choices for the corresponding pairs of numerators. By multiplying by appropriate units we may also assume without loss of generality that $s$ and  $s'$ lie in $\Z[i]^+$, so we have that
\begin{align}
  \Sigma_2&=\sum_{\substack{s\in\Z[i]^+\\|s|\le S}}\sum_{\substack{s'\in\Z[i]^+\\s'\text{cons to}~s}}\frac{1}{|s|^2|s'|^2}.\nonumber
\end{align}
Although the constant here is not important, for completeness we mention that we have also divided by an extra factor of $2$ to account for the fact that the double sum on the right hand side counts each pair $s,s'$ twice. Using Lemma \ref{lem.PowSum2} we now have that
\begin{align}
  \Sigma_2&\ll \sum_{\substack{s,s'\in\Z[i]^+\\|s|,|s'|\le S}}\frac{1}{|s|^2|s'|^2}\nonumber\\
  &\le \left(\sum_{\substack{|s|\le S}}\frac{1}{|s|^2}\right)^2\nonumber\\
  &\ll \log^2 S.\label{eqn.Sigma2Est1}
\end{align}
For $\Sigma_1$ we have that
\begin{align}
  \Sigma_1&=\sum_{\substack{s\in\Z[i]^+\\|s|\le S}}\sum_{\substack{s'\in\Z[i]^+\\s'\text{cons to}~s}}\frac{1}{|s|^4}\nonumber\\
  &=\frac{1}{4}\sum_{\substack{s\in\Z[i]^+\\|s|\le S}}\frac{1}{|s|^4}\sum_{\substack{s'\in\Omega_s\cap\Z[i]\\(s,s')=1}}1\label{eqn.DefSig1}.
  \end{align}
  Using M\"{o}bius inversion on the inner sum gives
  \begin{align*}
    \sum_{\substack{s'\in\Omega_s\cap\Z[i]\\(s,s')=1}}1&=\sum_{s'\in\Omega_s\cap\Z[i]}\sum_{d|s,s'}\mu_i(d)\\
    &=\sum_{d|s}\mu_i(d)\sum_{z\in d^{-1}\Omega_s}1.
  \end{align*}
  For each $d$, the number of lattice points in $d^{-1}\Omega_s$ is equal to the number of lattice points in the circle of radius $S/|d|$ centered at the origin, minus the number of lattice points in the intersection of the four translates of this circle by $\pm s/|d|$ and $\pm is/|d|$. The number of lattice points in each of these two convex regions can be calculated, via the Gauss circle method, using the machinery in \cite{Huxl2003}, and (see the comments immediately following \cite[Equation (1.1)]{Huxl2003}) the implied constants in the resulting error terms can be taken to be the same. This leads to the estimate
  \begin{equation*}
    \sum_{z\in d^{-1}\Omega_s}1=\frac{|\Omega_s|}{|d|^2}+O\left(\frac{S^{2\kappa}}{|d|^{2\kappa}}\right),
  \end{equation*}
  and using this in the equation above, we have that
  \begin{align}
    \sum_{\substack{s'\in\Omega_s\cap\Z[i]\\(s,s')=1}}1&=|\Omega_s|\sum_{d|s}\frac{\mu_i(d)}{|d|^2}+O\left(S^{2\kappa}\sum_{d|s}\frac{|\mu_i(d)|}{|d|^{2\kappa}}\right)\nonumber\\
    &=\frac{\phi_i(s)}{|s|^2}|\Omega_s|+O_\epsilon(S^{2\kappa+\epsilon}).\label{eqn.CoprimeLatPts2}
  \end{align}
  To briefly explain the estimate used in the error term here, first of all notice that
  \begin{equation*}
    \sum_{d|s}\frac{|\mu_i(d)|}{|d|^{2\kappa}}=\prod_{p|s}\left(1+\frac{1}{|p|^{2\kappa}}\right)\le\exp\left(c_1\sum_{p|s}\frac{1}{|p|^{2\kappa}}\right),
  \end{equation*}
  for some constant $c_1>0$. Using the prime number theorem for Gaussian integers (see the proof of \cite[Theorem 4.1]{Meas2018} for a more detailed explanation), there is a constant $c_2>0$ with the property that
  \begin{align*}
    \sum_{p|s}\frac{1}{|p|^{2\kappa}}\le\sum_{p^2\le c_2\log|s|}\frac{1}{|p|^{2\kappa}}\ll\frac{(\log |s|)^{1-2\kappa}}{\log\log|s|},
  \end{align*}
  and since
  \begin{equation*}
    \exp\left(\frac{(\log |s|)^{1-2\kappa}}{\log\log|s|}\right)\ll_\epsilon |s|^\epsilon,
  \end{equation*}
  for any $\epsilon>0,$ this explains the error term in \eqref{eqn.CoprimeLatPts2}.
  
  Returning to our estimate of $\Sigma_1$, we now have that
  \begin{align}
  \Sigma_1&=\frac{1}{4}\sum_{\substack{s\in\Z[i]^+\\|s|\le S}}\frac{\phi_i(s)}{|s|^6}|\Omega_s|+O_\epsilon\left(S^{2\kappa+\epsilon}\right)\nonumber\\
  &=\frac{1}{4}\Sigma_3+O_\epsilon(S^{2\kappa+\epsilon})\label{eqn.Sigma1Form1}.
\end{align}
Using formula \eqref{eqn.OmegaArea} we have that
\begin{equation*}
  \Sigma_3=\sum_{\substack{s\in\Z[i]^+\\|s|\le S}}\frac{\phi_i(s)}{|s|^6}I_1(|s|)-2\sum_{\substack{s\in\Z[i]^+\\|s|\le S}}\frac{\phi_i(s)}{|s|^4},
\end{equation*}
where
\begin{equation*}
  I_1(t)=8S^2\int_0^{\sin^{-1}\left(\frac{t}{\sqrt{2}S}\right)}\cos^2u~\mathrm{d}u.
\end{equation*}
From Lemma \ref{lem.PhiPowSum2} we have that
\begin{equation*}
  \sum\limits_{\substack{s\in\Z[i]^+\\|s|\le S}} \frac{\phi_i(s)}{|s|^4} = z_1\ln{S} + \frac{z_1 + z_2}{4} + O(S^{2\kappa-2}),
\end{equation*}
In order to estimate the other sum which appears above we first apply Lemma \ref{lem.AbelSumForm} with $x=S^6,~f(t)=1/t,$ and
\begin{equation*}
a_1(n)=\sum_{\substack{s\in\Z[i]^+\\|s|=n^{1/6}}}\phi_i(s)
\end{equation*}
to obtain
\begin{align*}
  \sum_{\substack{s\in\Z[i]^+\\|s|\le S}}\frac{\phi_i(s)}{|s|^6}&=\sum_{n\le x} a_1(n)f(n)\\
  &=S^{-6}\sum_{\substack{s\in\Z[i]^+\\|s|\le S}}\phi_i(s)+\int_1^{S^6}A_1(t)t^{-2}~\mathrm{d}t\\
  &=\frac{z_1}{S^2}+O(S^{2\kappa-4})+3z_1-\frac{3z_1}{S^2}+z_2'+O(S^{2\kappa-4})\\
  &=3z_1+z_2'-\frac{2z_1}{S^2}+O(S^{2\kappa-4}),
\end{align*}
with
\begin{equation*}
  z_2'=\int_1^\infty(A_1(t)-z_1t^{2/3})t^{-2}~\mathrm{d}t.
\end{equation*}
Next we apply Lemma \ref{lem.AbelSumForm} again, this time with $x=S^6,f(t)=I_1(t^{1/6}),$ and
\begin{equation*}
  a_2(n)=\frac{1}{n}\sum_{\substack{s\in\Z[i]^+\\|s|=n^{1/6}}}\phi_i(s).
\end{equation*}
We have that
\begin{align}
  A_2(t)=3z_1+z_2'-\frac{2z_1}{t^{1/3}}+O\left(t^{\frac{\kappa-2}{3}}\right),\label{eqn.ParSumEst3}
\end{align}
and that
\begin{align*}
  f'(t)&=8S^2\cos^2\left(\sin^{-1}\left(\frac{t^{1/6}}{\sqrt{2}S}\right)\right)\frac{\mathrm{d}}{\mathrm{d}t}\left(\sin^{-1}\left(\frac{t^{1/6}}{\sqrt{2}S}\right)\right)\\
  &=\frac{2\sqrt{2}}{3}St^{-5/6}\left(1-\frac{t^{1/3}}{2S^2}\right)^{1/2}.
\end{align*}
Therefore,
\begin{align}
  \int_{1}^{S^6}A_2(t)f'(t)~\mathrm{d}t&=\frac{2\sqrt{2}(3z_1+z_2')}{3}S\int_1^{S^6}t^{-5/6}\left(1-\frac{t^{1/3}}{2S^2}\right)^{1/2}~\mathrm{d}t\nonumber\\
  &\quad +\frac{2\sqrt{2}}{3}S\int_1^{S^6}t^{-5/6}(A_2(t)-(3z_1+z_2'))\left(1-\frac{t^{1/3}}{2S^2}\right)^{1/2}~\mathrm{d}t\nonumber\\
  &= X_1+X_2.\label{eqn.ACalc2}
\end{align}
Making the substitution $\sin u=t^{1/6}/(\sqrt{2}S)$, we find that
\begin{align*}
  X_1&=8(3z_1+z_2')S^2\int_{\sin^{-1}(1/\sqrt{2}S)}^{\pi/4}\cos^2u~\mathrm{d}u\\
  &=(3z_1+z_2')(\pi+2)S^2-4\sqrt{2}(3z_1+z_2')S+\frac{\sqrt{2}}{3S}+O(S^{-3}).
\end{align*}

Next let
\begin{equation}\label{eqn.z_2''Def}
z_2''=\frac{2\sqrt{2}}{3}\int_1^{\infty}t^{-5/6}(A_2(t)-(3z_1+z_2'))~\mathrm{d}t,
\end{equation}
which by  \eqref{eqn.ParSumEst3} is finite. Using a first order approximation for the function $(1-x)^{1/2}$ in the compact subregion $\{|x|\le 1/2\}$ of its interval of convergence, together with the estimate in \eqref{eqn.ParSumEst3}, we have that
\begin{align*}
&\frac{2\sqrt{2}}{3}\int_1^{S^6}t^{-5/6}(A_2(t)-(3z_1+z_2'))\left(1-\frac{t^{1/3}}{2S^2}\right)^{1/2}~\mathrm{d}t\\
&\qquad=\frac{2\sqrt{2}}{3}\int_1^{S^6}t^{-5/6}(A_2(t)-(3z_1+z_2'))~\mathrm{d}t+O\left(\frac{1}{S^2}\int_1^{S^6}t^{-5/6}~\mathrm{d}t\right)\\
&\qquad=z_2''+O\left(\frac{1}{S}\right),
\end{align*}
which proves that
\begin{align*}
  X_2=z_2''S+O\left(1\right).
\end{align*}

Substituting into \eqref{eqn.ACalc2} and using Lemma \ref{lem.AbelSumForm} now gives that
\begin{align*}
  \sum_{\substack{s\in\Z[i]^+\\|s|\le S}}\frac{\phi_i(s)}{|s|^6}I_1(|s|)&=A_2(S^6)f(S^6)-\int_1^{S^6}A_2(t)f'(t)~\mathrm{d}t\\
  &=\left(3z_1+z_2'-\frac{2z_1}{S^2}+O(S^{2\kappa-4})\right)(\pi+2)S^2-X_1-X_2\\
  &=\left(4\sqrt{2}(3z_1+z_2')-z_2''\right)S+O(1).
\end{align*}
Using these formulas in \eqref{eqn.2ndMomAsym1} and \eqref{eqn.Sigma1Form1} then gives the statement of Theorem \ref{thm.kthMom}, with
\begin{equation*}
  \xi_2=\sqrt{2}(3z_1+z_2')-\frac{z_2''}{4}.
\end{equation*}
Finally, in order to verify that $\xi_2> 0$, consider the contribution to $\Sigma_1$, as written in \eqref{eqn.DefSig1}, coming from the $s=1$ term. The coprimeness condition on the inner sum is automatically satisfied and, again using the machinery from \cite{Huxl2003} (see comments above relating to the uniformness of the constants in the error terms), we find that
\begin{align*}
\#\{s'\in\Omega_s\cap\Z[i]\}\ge \#\{s'\in\Z[i]:|s|\le S, |s-1|\ge S\}\gg S.
\end{align*}
Therefore $\mc{M}_2(S)\gg S$ and $\xi_2> 0$.

\section{Proof of Theorem \ref{thm.kthMom}, $k>2$ case}
Suppose now that $k>2$. First we write
\begin{align}
  \mc{M}_k(S)&=\frac{1}{2^k}\sum_{\frac{r}{s}\in\mc{G}_S}\frac{\#\{r'/s'\in \mc{G}_S ~\text{consecutive to}~ r/s\}}{|s|^{2k}}\nonumber\\
  &\quad +\frac{1}{2^k}\sum_{\ell=1}^{k-1}\binom{k}{\ell}\sum_{\substack{\frac{r}{s},\frac{r'}{s'}\in\mc{G}_S\\\text{consecutive}}}\frac{1}{|s|^{2\ell}|s'|^{2(k-\ell)}}\nonumber\\
  &=\Sigma_1+\Sigma_2.\label{eqn.kthMomAsym1}
\end{align}
Using Lemmas \ref{lem.PowSum2} and \ref{lem.PowSumTail1}, we have that
\begin{align}
  \Sigma_2&\ll_k\sum_{\ell=1}^{k-1}\sum_{\substack{s\in\Z[i]^+\\|s|\le S}}\sum_{\substack{s'\in\Z[i]^+\\s'\text{cons to}~s}}\frac{1}{|s|^{2\ell}|s'|^{2(k-\ell)}}\nonumber\\
  &\ll \sum_{\ell=1}^{k-1}\left(\sum_{\substack{s\in\Z[i]^+\\|s|\le S}}\frac{1}{|s|^{2\ell}}\right)\left(\sum_{\substack{s'\in\Z[i]^+\\|s'|\le S}}\frac{1}{|s'|^{2(k-\ell)}}\right) \nonumber\\ 
  &\ll_k \log S.\label{eqn.Sigma2EstKthMom}
\end{align}
This is asymptotically smaller than what we obtained in the $k=2$ case because the inner sum is divergent only if $\ell=1$ or $k-1$, and in either of these cases the other exponent appearing in the inner summand, $2(k-\ell)$ or $2\ell$, is at least $4$.

Next, using \eqref{eqn.CoprimeLatPts2} we have that
\begin{align}
  \Sigma_1&=\frac{1}{2^{k-2}}\sum_{\substack{s\in\Z[i]^+\\|s|\le S}}\sum_{\substack{s'\in\Z[i]^+\\s'\text{cons to}~s}}\frac{1}{|s|^{2k}}\nonumber\\  
  &=\frac{1}{2^{k}}\sum_{\substack{s\in\Z[i]^+\\|s|\le S}}\frac{1}{|s|^{2k}}\sum_{\substack{s'\in\Omega_s\cap\Z[i]\\(s,s')=1}}1\nonumber\\  
  &=\frac{1}{2^{k}}\sum_{\substack{s\in\Z[i]^+\\|s|\le S}}\frac{\phi_i(s)}{|s|^{2k+2}}|\Omega_s|+O_\epsilon\left(S^{2\kappa+\epsilon}\right)\nonumber\\
  &=\frac{1}{2^{k}}\Sigma_3+O_\epsilon(S^{2\kappa+\epsilon})\label{eqn.Sigma1Form2}.
\end{align}
As before, we write
\begin{equation*}
  \Sigma_3=\sum_{\substack{s\in \Z[i]^+ \\ |s| \leq S}}\frac{\phi_i(s)}{|s|^{2k+2}}I_1(|s|)-2\sum_{\substack{s\in \Z[i]^+ \\ |s| \leq S}}\frac{\phi_i(s)}{|s|^{2k}}.
\end{equation*}
From Lemma \ref{lem.PhiPowSum3} we have that
\begin{equation*}
  \sum_{\substack{s\in \Z[i]^+ \\ |s| \leq S}}\frac{\phi_i(s)}{|s|^{2k}}=z_k+O\left(\frac{1}{S^{2(k-2)}}\right),
\end{equation*}
and for the other sum we first apply Lemma \ref{lem.AbelSumForm} with $x=S^{2k+2},~f(t)=1/t,$ and
\begin{equation*}
a_1(n)=\sum_{\substack{s\in\Z[i]^+\\|s|=n^{1/(2k+2)}}}\phi_i(s)
\end{equation*}
to obtain
\begin{align*}
  \sum_{\substack{s\in \Z[i]^+ \\ |s| \leq S}}\frac{\phi_i(s)}{|s|^{2k+2}}&=\sum_{n\le x} a_1(n)f(n)\\
  &=S^{-2-2k}\sum_{\substack{s\in \Z[i]^+ \\ |s| \leq S}}\phi_i(s)+\int_1^{S^{2k+2}}A_1(t)t^{-2}~\mathrm{d}t\\
  &=\frac{z_1}{S^{2k-2}}+O(S^{2(\kappa-k)})+\left(\frac{k+1}{k-1}\right)z_1+O(S^{2(1-k)})+z_k'+O(S^{2(\kappa-k)})\\
  &=\left(\frac{k+1}{k-1}\right)z_1+z_k'+O(S^{2(1-k)}),
\end{align*}
with
\begin{equation*}
  z_k'=\int_1^\infty(A_1(t)-z_1t^{2/(k+1)})t^{-2}~\mathrm{d}t.
\end{equation*}
Next we apply Lemma \ref{lem.AbelSumForm} again, with $x=S^{2k+2},f(t)=I_1\left(t^{1/(2k+2)}\right),$ and
\begin{equation*}
  a_2(n)=\frac{1}{n}\sum_{\substack{s\in\Z[i]^+\\|s|=n^{1/(2k+2)}}}\phi_i(s).
\end{equation*}
We have that
\begin{align}
  A_2(t)=\left(\frac{k+1}{k-1}\right)z_1+z_k'+O\left(t^{\frac{1-k}{1+k}}\right),\label{eqn.ParSumEst4}
\end{align}
and that
\begin{align*}
  f'(t)&=8S^2\cos^2\left(\sin^{-1}\left(\frac{t^{1/(2k+2)}}{\sqrt{2}S}\right)\right)\frac{\mathrm{d}}{\mathrm{d}t}\left(\sin^{-1}\left(\frac{t^{1/(2k+2)}}{\sqrt{2}S}\right)\right)\\
  &=\frac{4\sqrt{2}}{2k+2}St^{-(2k+1)/(2k+2)}\left(1-\frac{t^{1/(k+1)}}{2S^2}\right)^{1/2}.
\end{align*}
Therefore,
\begin{align}
  &\int_{1}^{S^{2k+2}}A_2(t)f'(t)~\mathrm{d}t\nonumber\\
  &\qquad=\frac{4\sqrt{2}}{2k+2}\left(\left(\frac{k+1}{k-1}\right)z_1+z_k'\right)S\int_1^{S^{2k+2}}t^{-(2k+1)/(2k+2)}\left(1-\frac{t^{1/(k+1)}}{2S^2}\right)^{1/2}~\mathrm{d}t\nonumber\\
  &\qquad\quad +\frac{4\sqrt{2}}{2k+2}S\int_1^{S^{2k+2}}t^{-(2k+1)/(2k+2)}\left(A_2(t)-\left(\frac{k+1}{k-1}\right)z_1-z_k'\right)\left(1-\frac{t^{1/(k+1)}}{2S^2}\right)^{1/2}~\mathrm{d}t\nonumber\\
  &\qquad= X_1+X_2.\nonumber
\end{align}
Making the substitution $\sin u=t^{1/(2k+2)}/(\sqrt{2}S)$, we find that
\begin{align*}
  X_1&=8\left(\left(\frac{k+1}{k-1}\right)z_1+z_k'\right)S^2\int_{\sin^{-1}(1/\sqrt{2}S)}^{\pi/4}\cos^2u~\mathrm{d}u\\
  &=\left(\left(\frac{k+1}{k-1}\right)z_1+z_k'\right)(\pi+2)S^2-4\sqrt{2}\left(\left(\frac{k+1}{k-1}\right)z_1+z_k'\right)S+O\left(S^{-1}\right).
\end{align*}
Using \eqref{eqn.ParSumEst4}, we have that
\begin{align*}
  X_2=z_k''S+O\left(S^{-2k+3}\right)=z_k''S+O\left(S^{-2}\right),
\end{align*}
with
\begin{equation*}
  z_k''=\frac{4\sqrt{2}}{2k+2}\int_1^{\infty}t^{-(2k+1)/(2k+2)}\left(A_2(t)-\left(\frac{k+1}{k-1}\right)z_1-z_k'\right)~\mathrm{d}t.
\end{equation*}
This gives that
\begin{align*}
  \sum_{\substack{s\in \Z[i]^+ \\ |s| \leq S}}\frac{\phi_i(s)}{|s|^{2k+2}}I_1(|s|)&=A_2(S^{2k+2})f(S^{2k+2})-\int_1^{S^{2k+2}}A_2(t)f'(t)~\mathrm{d}t\\
  &=\left(\left(\frac{k+1}{k-1}\right)z_1+z_k'+O\left(S^{2-2k}\right)\right)(\pi+2)S^2-X_1-X_2\\
  &=\left(4\sqrt{2}\left(\left(\frac{k+1}{k-1}\right)z_1+z_k'\right)-z_k''\right)S+O(S^{-1}).
\end{align*}
Using these formulas gives the statement of Theorem \ref{thm.kthMom}, with
\begin{equation*}
  \xi_k=\frac{1}{2^{k}}\left(4\sqrt{2}\left(\left(\frac{k+1}{k-1}\right)z_1+z_k'\right)-z_k''\right).
\end{equation*}
As in the $k=2$ case, the contribution to $\Sigma_1$ from $s=1$ is large enough to guarantee that $\xi_k> 0$.

\vspace{.2in}

{\footnotesize

\noindent
AH: Department of Mathematics, University of Houston,\\
Houston, TX, United States.\\
haynes@math.uh.edu\\

\noindent
KM: Department of Mathematics, University of York,\\
Heslington, York, YO10 5DD, England\\
kekm501@york.ac.uk

}


\begin{thebibliography}{1}


\bibitem{Apos1990}
T.~M.~Apostol:
{\em Modular functions and Dirichlet series in number theory, Second edition,}
Graduate Texts in Mathematics, 41. Springer-Verlag, New York, 1990.

\vspace*{.1in}

\bibitem{BocaZaha2007}
F.~P.~Boca, A.~Zaharescu:
{\em The distribution of the free path lengths in the periodic two-dimensional Lorentz gas in the small-scatterer limit},
Comm. Math. Phys.  269  (2007),  no. 2, 425-471.

\vspace*{.1in}

\bibitem{ChauMaliZaha2015}
S.~Chaubey, A.~Malik, A.~Zaharescu:
{\em $k$-moments of distances between centers of Ford circles},
J. Math. Anal. Appl.  422  (2015),  no. 2, 906-919.


\vspace*{.1in}


\bibitem{Ford1938}
L.~R.~Ford:
{\em Fractions},
Amer. Math. Monthly  45  (1938),  no. 9, 586-601.

\vspace*{.1in}

\bibitem{Fran1924}
J.~Franel:
{\em Les suites de Farey et le probl\`{e}me des nombres premiers},
G\"{o}ttinger Nachrichten (1924), 198-201.

\vspace*{.1in}

\bibitem{Huxl1971}
M.~N.~Huxley:
{\em The distribution of Farey points I},
Acta Arith.  18  (1971), 281-287.

\vspace*{.1in}

\bibitem{Huxl2003}
M.~N.~Huxley:
{\em Exponential sums and lattice points III},
Proc. London Math. Soc. (3)  87  (2003),  no. 3, 591-609.

\vspace*{.1in}


\bibitem{Land1924}
E.~Landau:
{\em Bemerkungen zu der vorstehenden Abhandlung von Herrn Franel},
G\"{o}ttinger Nachrichten (1924), 202-206.

\vspace*{.1in}

\bibitem{Mark2010}
J.~Marklof:
{\em Horospheres and Farey fractions},
Dynamical numbers-interplay between dynamical systems and number theory, 
97-106, Contemp. Math., 532, Amer. Math. Soc., Providence, RI,  2010. 

\vspace*{.1in}

\bibitem{Meas2018}
K.~Measures:
{\em First moment of distances between centres of Ford spheres},
preprint, arXiv:1805.01508.

\vspace*{.1in}

\bibitem{Schi1972}
A.~Schinzel:
{\em Wac\l{}aw Sierpi\'{n}ski's papers on the theory of numbers},
Acta Arith.  21  (1972), 7-13.

\vspace*{.1in}

\bibitem{Sier1906}
W.~Sierpi\'{n}ski:
{\em O pewnem zagadneniu w rachunku funkcyj asymptotznych},
Prace Mat.-Fiz. 17 (1906), 77–118.



\end{thebibliography}
\end{document}